\documentclass{amsart}

\usepackage{amsmath,amssymb,amsthm,bbm}

\newtheorem{prop}{Proposition}[section]
\newtheorem{thm}[prop]{Theorem}

\newtheorem{cor}[prop]{Corollary}

\newtheorem*{thm2}{Theorem}

\theoremstyle{definition}

\newtheorem{rem}[prop]{Remark}
\newtheorem{defi}[prop]{Definition}

\newtheorem*{ack}{Acknowledgement}

\def\co{\colon\thinspace}

\newcommand{\C}{\mathbb C}

\newcommand{\rmd}{\mathrm d}
\newcommand{\D}{\mathbb D}

\newcommand{\rme}{\mathrm e}

\newcommand{\FF}{\mathcal F}

\newcommand{\Hp}{\mathbb H}

\newcommand{\II}{\mathcal I}
\newcommand{\rmi}{\mathrm i}

\newcommand{\JJ}{\mathcal J}

\newcommand{\MM}{\mathcal M}

\newcommand{\N}{\mathbb N}
\newcommand{\NN}{\mathcal N}

\newcommand{\OO}{\mathcal O}

\newcommand{\PP}{\mathcal P}

\newcommand{\R}{\mathbb R}
\newcommand{\RR}{\mathcal R}

\renewcommand{\SS}{\mathcal S}

\newcommand{\UU}{\mathcal U}

\newcommand{\WW}{\mathcal W}

\newcommand{\Z}{\mathbb Z}

\newcommand{\lra}{\longrightarrow}
\newcommand{\ra}{\rightarrow}

\DeclareMathOperator{\codim}{\mathrm{codim}}

\DeclareMathOperator{\dist}{\mathrm{dist}}

\DeclareMathOperator{\gl}{\mathrm{Gl}}

\DeclareMathOperator{\Int}{\mathrm{Int}}

\DeclareMathOperator{\jet}{\mathrm{Jet}}
\DeclareMathOperator{\Jet}{\JJ\!\mathrm{et}}

\DeclareMathOperator{\wst}{\rmd\mathbf{x}\wedge\rmd\mathbf{y}}

\begin{document}

\author{Kai Zehmisch}
\address{Mathematisches Institut, Westf\"alische Wilhelms-Universit\"at M\"unster}
\email{kai.zehmisch@uni-muenster.de}

\title[Holomorphic jets in symplectic manifolds]{Holomorphic jets in symplectic manifolds}

\date{14.04.2014}

\begin{abstract}
  We define pointwise partial differential relations
  for holomorphic discs.
  Given a relative homotopy class,
  a relation, and a generic almost complex structure
  we provide the moduli space of discs
  that have an injective point
  with the structure of a smooth manifold.
  Applications to the local behaviour are given
  and an inequality
  for the number of
  singularities is derived.
  Moreover, we show that
  for a coordinate class
  of a monotone Lagrangian split torus
  generically the number of
  non-immersed holomorphic discs is even.
\end{abstract}

\thanks{The author is partially supported by DFG grant
  ZE 992/1-1.}

\maketitle

%%%%%%%%%%%%%%%%%%%%%%%%%%%%%%%%%%%%%%%%%%%%%%%%%%%%%%%%%%%%%%%%%%%%%%

\section{Introduction\label{intro}}

Holomorphic curves are an important tool
of symplectic topology.
Many applications
to the geometry of contact
or symplectic manifolds
require pointwise intersection
conditions that include
finitely many derivatives.
We will give
a systematic introduction to
the theory of higher order
intersection problems.
As it turns out the main difficulties
appear in the study of intersection problems
that involve points on the boundary
of holomorphic discs.
To resolve these is the goal of this article.

We consider holomorphic curves $C$
in a symplectic manifold $(M,\omega)$
with respect to a compatible almost complex structure $J$.
The boundary $\partial C$
is contained in a maximally totally real
(e.g.\ Lagrangian) submanifold $L$.
We assume $C$ to be parametrized by a smooth map
$u\co (\Sigma,\Gamma)\ra(M,L)$
that is defined on a Riemann surface $\Sigma$
with boundary $\Gamma$.
The map $u$ is required to be {\bf holomorphic}
meaning that $u$ is a solution of
$J(u)\circ Tu=Tu\circ\rmi$,
where $\rmi$ denotes the complex structure
on the Riemann surface $\Sigma$,
see \cite{grom85}.

Locally holomorphic curves $u:\C\ra M$ exist
and the partial derivatives
\[\partial_xu(0),\ldots,\partial_x^ru(0)\]
in $x$-direction with respect to
conformal coordinates $z=x+\rmi y$
can take any given value
up to order $r=0,1,2\ldots$,
cf.\ \cite{sik94},
resp., Proposition \ref{repr}.
We will prove
the corresponding statement for
local holomorphic curves
$u\co (\Hp,\R)\ra (M,L)$
that are defined on the closed
upper half-plane
in Proposition \ref{holhalfrep} below.
Because $u$ is a homogeneous solution
of the non-linear Cauchy-Riemann equation
$u_x+J(u)u_y=0$
the partial derivatives in $x$-direction
determine the $r$-th Taylor polynomial
(which includes derivatives in $y$-direction as well)
uniquely at $0\in\C$.

We call an equivalence class of germs
of holomorphic maps
$u\co (\C,0)\ra (M,p)$,
resp.,
$u\co (\Hp,\R,0)\ra (M,L,p)$,
that have the same partial derivatives
in $x$-direction up to order $r$
a {\bf holomorphic $r$-jet}
on $(\Sigma,\Gamma)\times(M,L)$.
Identifying neighbourhoods of
$0\in\C$ with neighbourhoods
of $z\in\Sigma$,
resp.,
neighbourhoods of
$0\in\Hp$ with neighbourhoods
of $z\in\Gamma$ in $\Sigma$,
the point $0\equiv z$
is called the {\bf source}
of the holomorphic $r$-jet.
The space of all
holomorphic $r$-jets
is a locally trivial fibre bundle over
the so-called {\bf source space}
$(\Sigma,\Gamma)$,
see Section \ref{gentang}.
Given interior points
$z_1,\ldots,z_{m_0}\in\Sigma$,
boundary points $x_1,\ldots,x_{m_1}\in\Gamma$,
and non-negative integers
$r_1,\ldots,r_{m_0},s_1,\ldots,s_{m_1}$
we consider the product
of the fibres
over $z_k$
in the space of holomorphic $r_k$-jets
for $k=1,\ldots,m_0$
and the fibres
over $x_{\ell}$
in the space of holomorphic $s_{\ell}$-jets
for $\ell=1,\ldots,m_1$.
This product space
is denoted by $\Pi$.

Let $\JJ$ be the space
of almost complex structures
compatible with $\omega$
subject to restrictions that are
described in Section \ref{acstandfloer}.
The restrictions depend on the 
higher order intersection
problem under consideration.
Notice that the
construction of $\Pi$
involves the almost complex structure
$J\in\JJ$.
We call a submanifold
$R$ of $\Pi$
a {\bf holomorphic jet relation}
provided $R$ is independent
of the choice of $J$ in $\JJ$
in the sense of Definition \ref{def:hointrel}
and Section \ref{univreldefinition}.
Examples are defined by higher order tangency
and intersection relations
such as holomorphic curves that
\begin{itemize}
\item
  intersect a holomorphic submanifold,
  see \cite{ciemoh07},
\item
  intersect a Lagrangian submanifold
  (or cycles therein),
\item
  have double points or singularities,
  see \cite{bar99,bar00,ohzh09,wen10,ohzh11,oh11}.
\end{itemize}

An {\bf injective point} of a holomorphic map $u$
is an immersed point $z\in\Sigma$
such that the preimage of $u(z)$
under $u$ is $\{z\}$.
A holomorphic curve is called
{\bf somewhere injective}
if there exists an injective point
on each connected component.
In Section \ref{theunivjetsp}
(see Theorem \ref{mainthmfirstfrom}
and \ref{transforallsomewinjcurves})
we will prove that
the moduli space of
somewhere injective holomorphic curves
which represent a given homology class and jets in $R$
(see Section \ref{univmodspaceisbanmanifold}
and \ref{subsec:genpert} for the definition)
is a manifold
provided the almost complex structure
is chosen generically.

\begin{thm2}
  Let $R$ be a holomorphic jet relation
  that is closed as a subset of $\Pi$.
  Then there exists a subset
  $\JJ_{\mathrm{gen}}\subset\JJ$
  of second Baire category
  such that for all
  $J\in\JJ_{\mathrm{gen}}$
  the following holds:
  The moduli space
  of somewhere injective
  $J$-holomorphic maps
  $(\Sigma,\Gamma)\ra (M,L)$
  that represent a class
  $A\in H_2(M,L)$
  such that the $J$-holomorphic
  jets of orders
  $\mathbf{t}:=(r_1,\ldots,r_{m_0},s_1,\ldots,s_{m_1})$
  with sources
  $z_1,\ldots,z_{m_0},x_1,\ldots,x_{m_1}$
  (considered as a vector)
  take values in $R$
  is a smooth manifold.
  The dimension equals
  \[
  \mu(A)+
  n\big(\chi(\Sigma)-
  \|\mathbf{t}\|\big)+
  (1-n)\|\mathbf{m}\|+
  \dim R,
  \]
  where $\mu(A)$ is the Maslov number of $A$,
  $n$ half the dimension of $M$,
  $\chi(\Sigma)$ the Euler characteristic of $\Sigma$,
  $\|\mathbf{t}\|:=2(r_1+\ldots+r_{m_0})+s_1+\ldots+s_{m_1}$,
  and $\|\mathbf{m}\|:=2m_0+m_1$.
\end{thm2}

As an application we study the local behaviour
of holomorphic discs for generic almost complex structures.
Recall, that a holomorphic disc is called {\bf simple}
provided the set of injective points is dense.
A holomorphic disc is
{\bf simple along the boundary}
provided the set of injective points
on the boundary $\partial\D$
is dense in $\partial\D$,
see \cite{zehm13}.
For simple holomorphic discs
\cite[Corollary 8.5]{zehm13}
gives an alternative characterization.
Namely, a simple holomorphic disc $u$
is simple along the boundary
if and only if
the set of injective points of $u|_{\partial\D}$
is dense in $\partial\D$.

In view of the work of Lazzarini \cite{lazz00,lazz11}
and McDuff-Salamon \cite{mcsa04}
the most important difference
to the local behaviour of holomorphic spheres is
that the local covering number
need not to be constant along a holomorphic disc.
In particular, a somewhere injective
holomorphic disc is not simple in general.
But in the case
the almost complex structure is chosen generically
we will prove in Section \ref{aprioripertandlocb}
that a somewhere injective holomorphic disc
has a dense set of injective points in the interior
and on the boundary,
i.e., is simple and simple along the boundary.
In fact, the set of non-injective points is finite
if $n\geq 3$.
Further, in Section \ref{multicoverddiscs}
we extend Lazzarini's theorem
\cite[Theorem B]{lazz11},
which says that generically any non-constant
holomorphic disc is simple or multiply covered,
to the remaining dimension $4$.

In Section \ref{secadjineq} and \ref{seconsing}
we estimate the number of double points
and singularities (counted with multiplicity)
in terms of topological data.
In Section \ref{secexample} we give an example
how to define Gromov-Witten type invariants
that count discs with singular points.
In Section \ref{subsecemb} we discuss the generic existence
of immersed and embedded holomorphic curves.

%%%%%%%%%%%%%%%%%%%%%%%%%%%%%%%%%%%%%%%%%%%%%%%%%%%%%%%%%%%%%%%%%%%%%%
%%%%%%%%%%%%%%%%%%%%%%%%%%%%%%%%%%%%%%%%%%%%%%%%%%%%%%%%%%%%%%%%%%%%%%

\section{Generalized holomorphic tangencies\label{gentang}}

%%%%%%%%%%%%%%%%%%%%%%%%%%%%%%%%%%%%%%%%%%%%%%%%%%%%%%%%%%%%%%%%%%%%%%

\subsection{Definition\label{jetdef}}

Let $\Sigma$ be a Riemann surface
and $(M,J)$ a $2n$-dimensional almost complex manifold.
We consider holomorphic maps $u\co\Sigma\ra (M,J)$.
A {\bf local representation}
with {\bf source} $z\in\Sigma$
and {\bf target} $u(z)\in M$
consists of a conformal chart $(U,k)$
with $k(z)=0$
and a chart $(V,h)$
with $h\big(u(z)\big)=0$.
We can assume that $(h_*J)_{u(z)}=\rmi$,
where $\rmi$ denotes
the standard complex multiplication of $\C^n$.
Usually we will suppress
the localization in the notation.
Therefore,
the {\bf Cauchy-Riemann equation} gets
\[u_x+J(u)u_y=0\]
with respect to conformal coordinates $z=x+\rmi y$.

Let $r=0,1,2,\ldots$ be a non-negative integer.
We call two germs of holomorphic maps $u,v$
{\bf $r$-equivalent} at $z\in\Sigma$
if $u(z)=v(z)$
and if in a local representation
the partial derivatives in $x$-direction at $0$
coincide up to order $r$,
i.e., if
\[\partial_x^{\ell}u(0)=\partial_x^{\ell}v(0)\]
for all $\ell=1,\ldots,r$.
With the knowledge of
$\partial_xu(0),\ldots,\partial_x^ru(0)$
one can reconstruct the $r$-th Taylor polynomial at $0$
by taking partial derivatives of the Cauchy-Riemann equation
\[u_y=J(u)u_x\]
as follows:
\begin{eqnarray*}
  u_{xy}&=&J(u)u_{xx}+\big(DJ(u)\cdot u_x\big)u_x\\
  u_{yy}&=&J(u)u_{xy}+\big(DJ(u)\cdot u_y\big)u_x\\
  &\vdots&
\end{eqnarray*}
Therefore,
\emph{all} partial derivatives of $u$ and $v$
up to order $r$ coincide at $0$.
By the chain rule
this implies independence of the chosen local representation
which we used in the definition.
Therefore, the equivalence relation is well defined.
The $r$-equivalence class $j_z^ru$
is called the {\bf holomorphic $r$-jet}.
We denote the space of all $r$-jets by
\[(\Sigma\times M)_J^r\equiv\jet^r.\]

%%%%%%%%%%%%%%%%%%%%%%%%%%%%%%%%%%%%%%%%%%%%%%%%%%%%%%%%%%%%%%%%%%%%%%

\subsection{A representation}

For a small open neighbourhood $V$ of $0\in\R^{2n}$
the space of holomorphic jets
from $\C$ to $(V,J)$ can be identified with
\[(\C\times V)_J^r\equiv\C\times V\times (\R^{2n})^r\]
as we will prove in the following proposition.

\begin{prop}
  \label{repr}
  Let $J$ be an almost complex structure
  on $\R^{2n}$ with $J(0)=\rmi$.
  There exists an open neighbourhood
  $V$ of $0\in\R^{2n}$
  such that for all
  $(z_0,a_0)\in\C\times V$ and $a_1,\ldots,a_r\in\R^{2n}$
  there exists a germ of holomorphic maps
  $u\co\C\ra(\R^{2n},J)$
  with source $z_0$ and target $u(z_0)=a_0$
  satisfying
  \[\partial_xu(z_0)=a_1,\ldots,\partial_x^ru(z_0)=a_r.\]
\end{prop}

\begin{proof}
  Translations in $\C$ are conformal
  so that it is enough to prove the proposition for $z_0=0$.
  Let $p>2$ such that the elliptic regularity theory applies
  to the following argument,
  cf.\ \cite{mcsa04}.
  We consider the operator
  \[
  T(\xi)=
  \big(
  \xi_x+\rmi\xi_y;\,
  \xi(0),\partial_x\xi(0),\ldots,\partial_x^r\xi(0)
  \big).
  \]
  The domain consists of all $\xi\in W^{r+1,p}(D,\R^{2n})$,
  where $D$ is the unit disc,
  such that
  $\xi(\rme^{2\pi \rmi\theta})\in\rme^{(2r+1)\pi \rmi\theta}\R^n$
  for all $\theta\in[0,1)$,
  with $\R^n$ identified with $(\R\times\{0\})^n$.
  The target space is $W^{r,p}(D,\R^{2n})\times (\R^{2n})^{r+1}$.
  By \cite[Chapter C.4]{mcsa04} the operator $T$ is invertible.
  
  As on \cite[p. 627]{mcsa04} (written with Lazzarini)
  we consider the map
  \[
  F(v)=
  \big(
  v_x+J(v)v_y;\,
  v(0),\partial_xv(0),\ldots,\partial_x^rv(0)
  \big).
  \]
  The linearization at zero
  \[DF(0)\cdot\xi=T(\xi)\]
  is invertible.
  By the inverse function theorem $F$
  is a local diffeomorphism
  mapping zero in $W^{r+1,p}(D,\R^{2n})$
  to zero in
  $W^{r,p}(D,\R^{2n})\times (\R^{2n})^{r+1}$.
  Therefore, the equation
  $F(v)=(0;\,a_0,\varepsilon a_1,\ldots,\varepsilon^ra_r)$
  has a unique solution $v$
  for $a_0\in\R^{2n}$ close to the origin and
  $\varepsilon>0$ sufficiently small.
  The desired holomorphic germ is $u=v\circ 1/\varepsilon$.
\end{proof}

\begin{rem}
  \label{rem:repr}
  The solution $u=u(\mathbf{a})$ depends smoothly on
  $\mathbf{a}=(a_0,a_1,\ldots,a_r)$ for all $\mathbf{a}$
  contained in a sufficiently small neighbourhood
  of the origin in $\R^{r+1}$
  as the local diffeomorphism $F$
  in the proof of Proposition \ref{repr} shows.
\end{rem}

%%%%%%%%%%%%%%%%%%%%%%%%%%%%%%%%%%%%%%%%%%%%%%%%%%%%%%%%%%%%%%%%%%%%%%

\subsection{Differentiable structure}

Consider local representations $(U,k)$ and $(V,h)$.
By Proposition \ref{repr} and Remark \ref{rem:repr}
the maps
\[
\begin{array}{ccc}
  (U\times V)_J^r&\lra&(kU\times hV)_{h_*J}^r
  \\
  j_z^ru&\longmapsto&j_{k(z)}^r(h\circ u\circ k^{-1})
\end{array}
\]
define charts.
Therefore,
$(\Sigma\times M)_J^r$ is a manifold
of dimension $2\big(1+n(r+1)\big)$.
The projection onto $\Sigma\times M$
is an affine fibration,
cf.\ \cite{geig03}.

%%%%%%%%%%%%%%%%%%%%%%%%%%%%%%%%%%%%%%%%%%%%%%%%%%%%%%%%%%%%%%%%%%%%%%

\subsection{The case with boundary}

Consider a Riemann surface $\Sigma$
with boundary $\Gamma$.
The conformal atlas of $\Sigma$
is enriched by boundary preserving conformal maps
into the closed upper half-plane $\Hp$.
We consider holomorphic maps
which take values
in a maximally totally real
submanifold $L$ of $(M,J)$
along $\Gamma$.
By \cite[p. 539]{mcsa04}
there exists charts $h$ of $M$
about points on $L$
which take values in $\R^n$ along $L$
and satisfy $h_*J=\rmi$ along $\R^n$.
By a local representation
such a choice of charts is understood.

Two germs $u,v\co (\Sigma,\Gamma)\ra(M,L)$
of holomorphic maps at $x\in\Gamma$
define the same {\bf holomorphic $s$-jet}
($s=0,1,2,\ldots$)
provided that $u_{|\Gamma},v_{|\Gamma}$
are $s$-equivalent in the sense of
jets of smooth maps.
In other words,
the $s$-tangency class $j_x^ru$
is characterized as in Section \ref{jetdef}
with respect to local representations
that preserve the boundary condition.
The space of all holomorphic $s$-jets
on the boundary is denoted by
\[(\Gamma\times L)_J^s\equiv\jet^s.\]

\begin{prop}
  \label{holhalfrep}
  Any $s$-jet of a smooth map $\Gamma\ra L$
  has a holomorphic representative.
  Moreover,
  $(\Gamma\times L)_J^s$ is a manifold of dimension
  $1+n(s+1)$,
  and has the structure of an affine fibration
  over $\Gamma\times L$ that is
  induced by the source-target map.
\end{prop}

%%%%%%%%%%%%%%%%%%%%%%%%%%%%%%%%%%%%%%%%%%%%%%%%%%%%%%%%%%%%%%%%%%%%%%

\subsection{Holomorphic half-disc representation}

In this section
we will give a proof
of Proposition \ref{holhalfrep}.
Because the affine fibration structure
is obtained as in \cite{geig03}
it suffices to show
that $s$-jets of smooth maps $\Gamma\ra L$
can be represented by the restriction
of germs of $J$-holomorphic maps
$(\Sigma,\Gamma)\ra (M,L)$
such that the representation depends
smoothly on the data.
More precisely,
we consider an almost complex structure
$J$ on $\R^{2n}$
such that $J=\rmi$ on $\R^n$.
We claim that there exists
an open neighbourhood $V'$ of $0\in\R^n$
such that for all $(x_0,a_0)\in\R\times V'$
and $a_1,\ldots,a_s\in\R^n$
there exists a germ of $J$-holomorphic maps
$u\co(\Hp,\R)\ra(\R^{2n},\R^n)$ at $x_0$
satisfying
\[\partial_xu(x_0)=a_1,\ldots,\partial_x^su(x_0)=a_s,\]
such that $u$ depends
smoothly on $x_0,a_0,a_1,\ldots,a_s$.
It is enough to show this for $x_0=0$
because translations are holomorphic.

{\bf Step $\mathbf{1}$:}
On the unit disc $D\subset\C$
we consider the operator $u\mapsto u_x+\rmi u_y$
defined on the space $V$
consisting of all $\C^n$-valued functions
of Sobolev-class $W^{s+1,p}$, $p>2$,
subject to the boundary condition
$u(\rme^{2\pi \rmi\theta})\in\rme^{s\pi \rmi\theta}\R^n$,
$\theta\in [0,1)$.
The operator takes values in $W^{s,p}$.
By \cite[Chapter C.4]{mcsa04} this operator is onto
and its $n(s+1)$-dimensional kernel is generated by
\[
p_{\ell,b}(z)=
\frac{\rmi^{\ell}}{\ell !2^{s-\ell}}b(1+z)^{s-\ell}(1-z)^{\ell},
\]
$b\in\R^n, \ell=0,1,\ldots,s$.
Notice that
\[j_1^{\ell-1}p_{\ell,b}=0
\qquad\text{and}\qquad
\partial_y^{\ell}p_{\ell,b}(1)=b.\]
Set
\[
V_k=
\big\{
u\in V\;\big|\; D^{\alpha}u(1)=0\;\;\;\forall |\alpha|\leq k-1
\big\}
\]
and
\[
W_k=
\big\{
f\in W^{s,p}(D,\C^n)\;\big|\; D^{\beta}f(1)=0\;\;\;\forall |\beta|\leq k-2
\big\}
\]
for $k=0,1,\ldots,s$.
Due to the pointwise constraints and the boundary condition
the operator
\[
\begin{array}{rcc}
  S_k\co V_k&\lra& W_k\times\R^n
  \\
  u&\longmapsto&\big(u_x+\rmi u_y,\partial_y^ku(1)\big)
\end{array}
\]
is well defined,
onto, and its $n(s-k)$-dimensional kernel
is generated by $p_{\ell,b}$
for $\ell=k+1,\ldots,s$ and $b\in\R^n$.

{\bf Step $\mathbf{2}$:}
Let $\Omega$ be a domain in $\Hp$
obtained by smoothing the corners
of the unit half-disc
such that $\partial\Omega\cap\R$
is an interval $I$ which contains $0$.
Let $\varphi$ be a conformal diffeomorphism
$(\Omega,0)\ra (D,1)$ up to the boundary.
Let $V_{\Omega}$ be the space of all
$\C^n$-valued functions
of Sobolev-class $W^{s+1,p}$ on $\Omega$
subject to the boundary condition
\[u(z)\in\rme^{s\tfrac{\rmi}{2}\arg\varphi(z)}\R^n,\]
$z\in\partial\Omega$,
where the argument is normalized by
$\arg(\rme^{2\pi \rmi\theta})=2\pi\theta$.
Abbreviate $W^{s,p}(\Omega,\C^n)$ by $W_{\Omega}$.
Define
\[
V_{\Omega,k}=
\big\{
u\in V_{\Omega}\;\big|\; D^{\alpha}u(0)=0\;\;\;\forall |\alpha|\leq k-1
\big\}
\]
and
\[
W_{\Omega,k}=
\big\{
f\in W_{\Omega}\;\big|\; D^{\beta}f(0)=0\;\;\;\forall |\beta|\leq k-2
\big\}
\]
and
\[
\begin{array}{rcc}
  S_{\Omega,k}\co V_{\Omega,k}&\lra& W_{\Omega,k}\times\R^n
  \\
  u&\longmapsto&\big(u_x+\rmi u_y,\partial_x^ku(0)\big).
\end{array}
\]
The invertible maps $u\mapsto u\circ\varphi$ and
\[
(f,h)\longmapsto
\big(
(\varphi_x^1+\rmi\varphi_y^1)f\circ\varphi,c^kh
\big)
\]
(using $\varphi=\varphi^1+\varphi^2$
and $\partial_x\varphi(0)=c\rmi$
for a positive constant $c$)
conjugate $S_{\Omega,k}$ to $S_k$.
Therefore, $S_{\Omega,k}$
is onto and has $n(s-k)$-dimensional kernel.

{\bf Step $\mathbf{3}$:}
Let $\beta\co\C\ra\C$ be a smooth function
such that $\beta$ is identically $1$ on the
interval $I$.
The assignment
\[
\partial\Omega\ni z
\longmapsto
L(z):=\beta(z)\R^n
\]
defines a smooth loop of totally real
subspaces of $\C^n$ that splits up into real line bundles
with Maslov index $s$ corresponding to
$\C\times\ldots\times\C$.
Moreover, $L(x)=\R^n$ for $x\in I$.
Because the Maslov index of $L$ is $ns$
there exists a smooth function $A\co\C\ra\gl(n,\C)$
such that $A(0)=\mathbbm{1}$ and
\[L(z)=A(z)\,\rme^{s\tfrac{\rmi}{2}\arg\varphi(z)}\R^n\]
for all $z\in\partial\Omega$,
cf.\ Arnol'd's theorem on \cite[p. 554]{mcsa04}.

Consider the operator
\[
\hat{Q}\co
V_{\Omega}\lra W_{\Omega}\times(\R^n)^{s+1}
\]
mapping
\[
u\longmapsto
\big(
u_x+\rmi u_y;\,
u(0),\partial_x(Au)(0),\ldots,\partial_x^s(Au)(0)
\big).
\]
Observe that
$\partial_x(Au)(0)=A_x(0)u(0)+u_x(0)$
for example.
In order to show surjectivity of
$\hat{Q}$ let $f\in W_{\Omega}$
and $b_0,b_1,\ldots,b_s\in\R^n$
be arbitrarily given.
Because $S_{\Omega,0}$ is onto
we find $v^0\in V_{\Omega}$
such that $v_x^0+\rmi v_y^0=f$ and $v^0(0)=b_0$.
Assume that $v^{\ell}\in V_{\Omega}$
is constructed inductively
such that $v_x^{\ell}+\rmi v_y^{\ell}=f$ and
\[v^{\ell}(0)=b_0,\ldots,\partial_x^{\ell}(Av^{\ell})(0)=b_{\ell}.\]
Considering the operator $S_{\Omega,\ell+1}$
we find a holomorphic $q\in V_{\Omega,\ell+1}$
such that
\[
\partial_x^{\ell+1}q(0)=b_{\ell+1}-\partial_x^{\ell+1}(Av^{\ell})(0).
\]
Set $v^{\ell+1}=v^{\ell}+q$
and observe that
$v_x^{\ell+1}+\rmi v_y^{\ell+1}=f$ and
\[
v^{\ell+1}(0)=b_0,\ldots,\partial_x^{\ell+1}(Av^{\ell+1})(0)=b_{\ell+1}.
\]
Therefore,
$u=v^s$ is a solution of
$\hat{Q}(u)=\big(f;\,b_0,\ldots,b_s\big)$,
i.e., $\hat{Q}$ is onto.

{\bf Step $\mathbf{4}$:}
We show that $\hat{Q}$ has trivial kernel.
Consider $u\in\ker\hat{Q}$.
If $u$ is not constant
so is at least
one of the coordinate functions $u^j\in\C$.
By \cite[Lemma 9.5]{geizeh10}
the sum of the orders of the zeros
of $u^j$ on the boundary
plus twice the orders of zeros
of $u^j$ at the interior equals the Maslov index $s$.
But $u^j$ vanishes
at least to the $(s+1)$-st order at zero.
This implies $u=0$.
Therefore, $\hat{Q}$ is invertible.

{\bf Step $\mathbf{5}$:}
Let $V_L$ be the space of all
$\C^n$-valued functions of
Sobolev-class $W^{s+1,p}$ on $\Omega$
such that $u(z)\in L(z)$
for all $z\in\partial\Omega$.
Consider the operator
\[
T\co V_L\lra W_{\Omega}\times(\R^n)^{s+1}
\]
mapping
\[
u\longmapsto
\big(
u_x+\rmi u_y;\,u(0),\partial_xu(0),\ldots,\partial_x^su(0)
\big).
\]
The operator
\[
Q=(A^{-1}\times\mathbbm{1})\circ T\circ A\co
V_{\Omega}\lra W_{\Omega}\times(\R^n)^{s+1}
\]
equals
\[
Q(u)=\hat{Q}(u)+\big(A^{-1}(A_x+\rmi A_y)u;\,0\ldots,0\big).
\]
The second term is a compact perturbation.
Because $\hat{Q}$ is invertible
the operator $Q$ is Fredholm of index zero.
Hence, $T$ is a Fredholm operator of index zero.
Because $L$ splits into real line bundles of Maslov index $s$
corresponding to the complex coordinate planes
an application of
\cite[Lemma 9.5]{geizeh10} as in Step $4$
shows triviality of $\ker T$.
Hence, the operator $T$ is invertible.

{\bf Step $\mathbf{6}$:}
Consider the map
\[
F\co V_L\lra W_{\Omega}\times(\R^n)^{s+1}
\]
defined by
\[
F(v)=
\big(
v_x+J(v)v_y;\,
v(0),\partial_xv(0),\ldots,\partial_x^sv(0)
\big).
\]
Taking the derivative of at $v=0$
proves the Proposition \ref{holhalfrep}
similar to the last step in the proof of Proposition \ref{repr}.
\hfill Q.E.D.

%%%%%%%%%%%%%%%%%%%%%%%%%%%%%%%%%%%%%%%%%%%%%%%%%%%%%%%%%%%%%%%%%%%%%%
%%%%%%%%%%%%%%%%%%%%%%%%%%%%%%%%%%%%%%%%%%%%%%%%%%%%%%%%%%%%%%%%%%%%%%

\section{The universal jet space\label{theunivjetsp}}

We consider a
$2n$-dimensional symplectic manifold
$(M,\omega)$,
a compatible almost complex structure $J_0$
with the induced metric $\omega(\,.\,,J_0\,.\,)$, and
a maximally totally real submanifold $L$.

%%%%%%%%%%%%%%%%%%%%%%%%%%%%%%%%%%%%%%%%%%%%%%%%%%%%%%%%%%%%%%%%%%%%%%

\subsection{Almost complex structures\label{acstandfloer}}

The space $\JJ$ of all
compatible almost complex structures $J$
is a Fr\'echet manifold.
The model is
the space of symmetric
endomorphism fields $Y$ on $M$
which anti-commute with $J_0$.
The homeomorphism
\[Y\longmapsto J_0(\mathbbm{1}+Y)(\mathbbm{1}-Y)^{-1},\]
$\|Y\|_{C^0}<1$,
serves as a global parametrization for $\JJ$.
The inverse is
\[H\co J\longmapsto (J+J_0)^{-1}(J-J_0),\]
cf.\ \cite[Proposition 1.1.6]{aud94}.
For global considerations
we require in addition
the almost complex structures
to coincide with $J_0$
in the complement of a relatively compact
subset of $M$,
the so-called {\bf perturbation domain}.
If for example one considers
intersection problems of holomorphic curves with
$J_0$-holomorphic or maximally totally real submanifolds
with respect to $J_0$
one requires the manifolds
to be contained in the complement
of the perturbation domain.

Notice that
Lagrangian submanifolds
stay totally real after
a perturbation of the
almost complex structure $J$
provided $J$ is compatible
(or tamed) by $\omega$.
To achieve the same effect
for maximally totally real
submanifolds
one restricts to perturbations
that are $C^0$-close to $J_0$.
One way to formalize this
is the following construction
taken from \cite{geizeh13}:
Choose a metric $g$ on $M$
such that $J_0$ is orthogonal
and $J_0$ maps each tangent space
of $L$ to its orthogonal complement.
Notice that the restriction
of $g(J_0\,.\,,.\,)$ to $TL$ vanishes.
Following the arguments of
Weinstein's Lagrangian neighbourhood theorem,
see \cite[Theorem 3.33]{mcsa98},
one obtains a diffeomorphism
of a neighbourhood of $L$
onto a neighbourhood of the
zero section in $T^*L$
such that the pull-back $\omega_L$
of the Liouville symplectic form
coincides with $g(J_0\,.\,,.\,)$
on $TM|_L$.
Notice that $L$ is Lagrangian
with respect to $\omega_L$
and that $\omega_L$ tames $J_0$
in a neighbourhood of $L$.
Therefore,
we require in addition
all $J\in\JJ$ to be tamed by $\omega_L$
on the submanifold $L$.

In order to obtain a Banach manifold
of almost complex structure
we restrict to a further subset:
For a sequence $(\varepsilon_j)$
of positive real numbers
we consider infinitesimal
almost complex structures
that are finite in the {\bf Floer-norm}
\[\sum_{j=0}^{\infty}\varepsilon_j\|Y\|_{C^j}.\]
We choose $(\varepsilon_j)$
such that the resulting open subset is separable
and dense in $C^{\infty}$,
cf.\ \cite{sch95}.
The image under $H^{-1}$
is the separable Banach manifold $\II$
of compatible almost complex structures
with the induced {\bf Floer-topology}.

%%%%%%%%%%%%%%%%%%%%%%%%%%%%%%%%%%%%%%%%%%%%%%%%%%%%%%%%%%%%%%%%%%%%%%

\subsection{Definition}

We consider interior points
$z_1,\ldots,z_{m_0}$ on $\Sigma$
and
boundary points
$x_1,\ldots,x_{m_1}$ on $\Gamma$.
For $k=1,\ldots,m_0$
let $u^k\co\Sigma\ra M$
be a germ of a smooth maps
with source $z_k$.
For $\ell=1,\ldots,m_1$
let $v^{\ell}\co\Gamma\ra L$
be a germ of a smooth maps
with source $x_{\ell}$,
where we think of the $v^{\ell}$
to be extended to a
smooth map from $\Sigma$
into $M$.
We will assume that the points
$\{z_1,\ldots,z_{m_0}\}$
and $\{x_1,\ldots,x_{m_1}\}$
are pairwise distinct.
Therefore,
the maps
$u^1,\ldots,u^{m_0},v^1,\ldots,v^{m_1}$
are restrictions of a
smooth map
$u\co(\Sigma,\Gamma)\ra(M,L)$
to a neighbourhood of
the source points
as indicated by the labels.
In other words,
$u$ represents a germ of maps
near the closed subset
$\{z_1,\ldots,z_{m_0},x_1,\ldots,x_{m_1}\}$
of $\Sigma$.
Furthermore,
about each of the points
$p_k=u^k(z_k)$,
resp.,
$p_{m_0+\ell}=v^{\ell}(x_{\ell})$,
we consider a germ
of almost complex structures
$J^k\in\JJ$, $k=1,\ldots,m_0$,
resp.,
$J^{m_0+\ell}\in\JJ$,
$\ell=1,\ldots,m_1$.
We remark
that at equal points
$p_{j_1}=p_{j_2}$,
$j_1\neq j_2$,
the almost complex structures
$J^{j_1}$ and $J^{j_2}$
may differ.
This means that the $J^j$'s
are \emph{not} in general
restrictions
of a global almost complex structure
$J\in\JJ$.
But in all our applications
such a global $J\in\JJ$ will exist.

In order to shorten the notation
we will use the following conventions:
We call the tuple
\[
(\mathbf{z},\mathbf{x}):=
(z_1,\ldots,z_{m_0},x_1,\ldots,x_{m_1})
\]
of pairwise distinct points
a {\bf source}.
The space of all such tuples,
the {\bf source space},
is denoted by $\SS$.
It is a subset of
\[
\SS\subset
(\Sigma,\Gamma)^{\mathbf{m}}
:=
\Sigma^{m_0}\times\Gamma^{m_1},\]
where we wrote
$\mathbf{m}=(m_0,m_1)$
for the number of source points.
The {\bf length}
by definition is
$\|\mathbf{m}\|=2m_0+m_1$.
Moreover, we write
\[
(M,L)^{\mathbf{m}}
:=
M^{m_0}\times L^{m_1}
\]
so that
\[
u^{\mathbf{m}}:=
(u^1,\ldots,u^{m_0},v^1,\ldots,v^{m_1})
\]
represents a germ of smooth maps from
$(\Sigma,\Gamma)^{\mathbf{m}}$ into
$(M,L)^{\mathbf{m}}$.
We call $u^{\mathbf{m}}$ a {\bf multi-germ}
with source $(\mathbf{z},\mathbf{x})\in\SS$.
Similarly,
a {\bf multi-germ} of almost complex structures
\[
J^{\mathbf{m}}:=
\big(
J^1,\ldots,J^{m_0},
J^{m_0+1},\ldots,J^{m_0+m_1}
\big)
\]
with source
\begin{eqnarray*}
  u(\mathbf{z},\mathbf{x})
  &:=&
  \big(
  u^1(z_1),\ldots,u^{m_0}(z_{m_0}),
  v^1(x_1),\ldots,v^{m_1}(x_{m_1})
  \big)
  \\
  &=&
  \big(
  p_1,\ldots,p_{m_0},p_{m_0+1},\ldots,p_{m_0+m_1}
  \big)
  \\
  &=:&
  \mathbf{p}
\end{eqnarray*}
is understood.
Observe, that the tuple $\mathbf{p}$
is \emph{not} required to consists of pairwise distinct points.
The space of all multi-germs of almost complex structures
is denoted by $\JJ^{\mathbf{m}}$.

We say that
\[
(u,J)^{\mathbf{m}}:=(u^{\mathbf{m}},J^{\mathbf{m}})
\]
is a {\bf holomorphic multi-germ}
with source $(\mathbf{z},\mathbf{x})$
provided that
$J^{\mathbf{m}}\in\JJ^{\mathbf{m}}$
is a multi-germ
with source
$u(\mathbf{z},\mathbf{x})$
and the coordinate maps
of $u^{\mathbf{m}}$
are germs of holomorphic maps
with source $(\mathbf{z},\mathbf{x})$
that are holomorphic with respect to
the entries of the tuple $J^{\mathbf{m}}$.
We think of $u^{\mathbf{m}}$
as a smooth map
$u\co (\Sigma,\Gamma)\ra (M,L)$,
whose restriction to
disjoint neighbourhoods of
the elements of the set
$\{\mathbf{z},\mathbf{x}\}$
are holomorphic with respect to
the (local) almost complex structures
as indicated by the entries of $J^{\mathbf{m}}$.

On the set of holomorphic multi-germs
$(u,J)^{\mathbf{m}}$ we define a
{\bf $(\mathbf{r},\mathbf{s})$-tangency relation}:
Consider a vector of non-negative integers
\[
\mathbf{t}:=
(\mathbf{r},\mathbf{s}):=
(r_1,\ldots,r_{m_0},s_1,\ldots,s_{m_1}).
\]
The {\bf length} is
\[
\|\mathbf{t}\|:=
2(r_1+\ldots+r_{m_0})+s_1+\ldots+s_{m_1}.
\]
Two holomorphic multi-germs
$(u_1,J_1)^{\mathbf{m}}$ and
$(u_2,J_2)^{\mathbf{m}}$
with the same source
$(\mathbf{z},\mathbf{x})$
are said to be
{\bf $\mathbf{t}$-equivalent}
provided that
\begin{itemize}
\item 
  $
  u_1(\mathbf{z},\mathbf{x})=
  u_2(\mathbf{z},\mathbf{x})=:\mathbf{p}
  $\;,
\item 
  the 
  $(\mathbf{t}-\mathbf{1})${\bf-jets}
  of $J_1^{\mathbf{m}}$ and $J_2^{\mathbf{m}}$ at $\mathbf{p}$
  coincide,
  i.e.,
  $j_{\mathbf{p}}^{\mathbf{t}-\mathbf{1}}J_1^{\mathbf{m}}=
  j_{\mathbf{p}}^{\mathbf{t}-\mathbf{1}}J_2^{\mathbf{m}}$,
  where we used the notation
  \[
  \qquad\quad
  j_{\mathbf{p}}^{\mathbf{t}-\mathbf{1}}J^{\mathbf{m}}=
  \big(
  j_{p_1}^{r_1-1}J^1,\ldots,j_{p_{m_0}}^{r_{m_0}-1}J^{m_0},
  j_{p_{m_0+1}}^{s_1-1}J^{m_0+1},\ldots,j_{p_{m_0+m_1}}^{s_{m_1-1}}J^{m_0+m_1}
  \big)
  \]
  of smooth jets,
  and
\item
  $
  j_{(\mathbf{z},\mathbf{x})}^{\mathbf{t}}u_1=
  j_{(\mathbf{z},\mathbf{x})}^{\mathbf{t}}u_2
  $,
  where we used the notation
  \[
  j_{(\mathbf{z},\mathbf{x})}^{\mathbf{t}}u=
  (j_{\mathbf{z}}^{\mathbf{r}}u,
  j_{\mathbf{x}}^{\mathbf{s}}u)=
  \big(
  j_{z_1}^{r_1}u,\ldots,j_{z_{m_0}}^{r_{m_0}}u,
  j_{x_1}^{s_1}u,\ldots,j_{x_{m_1}}^{s_{m_1}}u
  \big)
  \]
  of holomorphic jets.
\end{itemize}
As in Section \ref{jetdef}
one verifies that this is well defined.
The $\mathbf{t}$-equivalence class
of a holomorphic multi-germ $(u,J)^{\mathbf{m}}$
is denoted by
$j_{(\mathbf{z},\mathbf{x})}^{\mathbf{t}}(u,J)$.
The notion of a $\mathbf{t}$-tangency class
is use synonymously.
The space of all
$\mathbf{t}$-equivalence classes
$j_{(\mathbf{z},\mathbf{x})}^{\mathbf{t}}(u,J)$
is denoted by
\[
\Jet^{\mathbf{t}}.
\]

%%%%%%%%%%%%%%%%%%%%%%%%%%%%%%%%%%%%%%%%%%%%%%%%%%%%%%%%%%%%%%%%%%%%%%

\subsection{An affine fibration\label{subsec:afffib}}

Let $E$ be the bundle over $M$ whose fibre at $p\in M$
consists of $\omega_p$-compatible complex structures
on the tangent space $T_pM$.
Observe, that the space of all smooth sections into $E$
coincides with $\JJ$.
We denote by $E^{(\mathbf{t}-\mathbf{1})}$
the space of $(\mathbf{t}-\mathbf{1})$-jets
of smooth sections into the bundle $E$.

\begin{prop}
  \label{univdiffstr}
  The natural map
  \[
  j_{(\mathbf{z},\mathbf{x})}^{\mathbf{t}}(u,J)
  \longmapsto
  \big(
  (\mathbf{z},\mathbf{x}),
  j_{u(\mathbf{z},\mathbf{x})}^{\mathbf{t}-\mathbf{1}}J^{\mathbf{m}}
  \big)
  \]
  is a locally trivial fibration
  \[\Jet^{\mathbf{t}}\lra\SS\times E^{(\mathbf{t}-\mathbf{1})}\]
  with affine fibre of dimension $n\|\mathbf{t}\|$.
\end{prop}

\begin{proof}
  With help of a Hermitian trivialization
  and \cite[Proposition 1.1.6]{aud94}
  the fibre of $E\ra M$
  is diffeomorphic to the unit ball $W$
  in the $n(n+1)$-dimensional vector space
  of symmetric $(2n\times 2n)$-matrices
  that anti-commute with $\rmi$,
  cf.\ with the beginning of Section \ref{acstandfloer}.
  Hence, $E^{r-1}$ is modeled on
  $V\times W\times\R^{d_{r-1}}$
  (resp.\ $V'\times W\times\R^{d_{r-1}}$)
  with
  \[
  d_{r-1}=\frac{(2n+r-1)!}{(2n)!(r-1)!}n(n+1),
  \]
  where $V\subset\R^{2n}$
  (resp.\ $V'\subset\R^n$)
  is an open subset.
  As in Section \ref{gentang}
  the space of holomorphic
  $r$-jets (for fixed $J$)
  can be described locally by
  $U\times V\times \R^{2rn}$
  (resp.\ $U'\times V'\times \R^{rn}$)
  for open subsets $U\subset\Sigma$
  (resp.\ $U'\subset\Gamma$).
  
  We claim that
  the local model for $\Jet^{\mathbf{t}}$ is
  \[
  (U,U')^{\mathbf{m}}\times
  (V,V')^{\mathbf{m}}\times
  (\R^{2n},\R^n)^{\mathbf{t}}\times
  (W\times\R^{d_{r_1-1}})\times
  \ldots\times
  (W\times\R^{d_{s_{m_1}-1}}).
  \]
  Considering each $\mathbf{m}$-coordinate separately
  we represent jets of $J\in\JJ$
  by local $\omega$-compatible almost complex structures.
  For this we use Taylor polynomials
  and a uniform cut-off function
  in the model space $H(\JJ)$ of $\JJ$.
  Represent a holomorphic jet by a germ of holomorphic maps
  as in Section \ref{gentang}
  with respect to the constructed local $J$'s.
  Notice that a smooth variation of jets of $J\in\JJ$
  is followed by a smooth variation of local $J$'s.
  Moreover,
  a smooth variation of $\mathbf{a}=(a_0,a_1,\ldots,a_r)$
  induces a smooth variation of local holomorphic maps:
  We define
  \[
  \FF(\mathbf{b},v,J)=
  \big(
  v_x+J(v)v_y;
  \,v(0)-b_0,\partial_xv(0)-b_1,\ldots,\partial_x^rv(0)-b_r
  \big).
  \]
  In view of Proposition \ref{repr} and \ref{holhalfrep}
  the partial derivative
  \[
  \frac{\partial\FF}{\partial v}(\mathbf{b},0,J)\cdot\xi=T(\xi)
  \]
  is invertible.
  By the implicit function theorem
  there exists a local smooth map $v=v(\mathbf{b},J)$
  such that
  \[
  \FF\big(\mathbf{b},v(\mathbf{b},J),J\big)=0.
  \]
  The local map
  \[
  u
  \big(
  a_0,a_1,\ldots,a_r,J
  \big)(z)
  =
  v
  \big(
  a_0,\varepsilon a_1,\ldots,\varepsilon^ra_r,J
  \big)(z/\varepsilon)
  \]
  is $J$-holomorphic with $j_0^ru=\mathbf{a}$
  and varies smoothly with the data
  (including the variation
  given by translations
  of the source $z_0$ (resp.\ $x_0$)).
  
  The argument requires $J(0)=\rmi$
  (resp.\ $J=\rmi$ on $\R^n$) for $J\in\JJ(V)$
  and open subsets $V\subset\R^{2n}\equiv M$.
  In order to achieve this we
  take a symplectic Darboux chart in the interior case.
  Along the boundary we use a Weinstein chart
  that is constructed as follows:
  By Weinstein's neighbourhood theorem
  there exists a local symplectic embedding
  $(T^*L,\OO_{T^*L})\ra(M,L)$.
  If $L$ is not Lagrangian with respect to $\omega$
  we use the symplectic form
  $\omega_L$ constructed in Section \ref{acstandfloer}.
  Consider a local parametrization $\R^n\ra L$
  of $L$, which induces a symplectic embedding
  $(T^*\R^n,\R^n)\ra(T^*L,\OO_{T^*L})$
  in a natural way.
  The inverse of the composition $(T^*\R^n,\R^n)\ra(M,L)$
  restricted to the image is the desired
  Weinstein chart.
  
  Set
  \[
  \varphi_J(\mathbf{x},\mathbf{y})=
  x^j\partial_{x^j}+y^jJ(\mathbf{x},\mathbf{y})\partial_{x^j}.
  \]
  Because $J$ is compatible with
  (resp.\ tamed by) $\wst$
  and the $\partial_{x^j}$'s
  span the Lagrangian $\R^n$ about $0$
  (resp.\ on $\R^n$)
  the vector fields
  $\partial_{x^j},J(\mathbf{x},\mathbf{y})\partial_{x^j}$,
  $j=1,\ldots,n$, form a basis near the origin.
  The differential is
  $(\rmi,J)$-complex and invertible at $0$
  (resp.\ on $\R^n$).
  Transforming the position
  and the almost complex structures via
  \[
  (\mathbf{x},\mathbf{y},J)
  \mapsto
  \big(
  \varphi_J(\mathbf{x},\mathbf{y}),(\varphi_J)_*J
  \big)
  \]
  allows the use of the map $\FF$.
  This shows smoothness
  of the transition function
  which are affine, see \cite{geig03}.
\end{proof}

\begin{rem}
  \label{tangentvector}
  The proof yields a description
  of a tangent vector of
  $\Jet^{\mathbf{t}}$
  at a point
  $j_{(\mathbf{z},\mathbf{x})}^{\mathbf{t}}(u,J)$.
  Let $Y$ be a
  local infinitesimal almost complex structure
  in $T_J\JJ$
  and $\xi$ a variation of
  local holomorphic maps.
  That is,
  $\xi$ is a local vector field along $u$
  such that
  $(\mathbf{v},\mathbf{w})\in T_{(\mathbf{z},\mathbf{x})}\SS$,
  $\tilde{\xi}_{\mathbf{m}}=\xi_{\mathbf{m}}+Du_{\mathbf{m}}\cdot
  (\mathbf{v},\mathbf{w})$, and
  $\tilde{Y}=Y+DJ(u)\cdot\tilde{\xi}$
  solve
  \[
  \tilde{\xi}_x+J(u)\tilde{\xi}_y+\tilde{Y}(u)u_y
  =0.
  \]
  Hence, a tangent vector can be represented
  by the (in the sense of the above equation) holomorphic
  $\mathbf{t}$-equivalence class
  $j_{(\mathbf{v},\mathbf{w})}^{\mathbf{t}}(\tilde{\xi},Y)$
  of a multi-germ $(\tilde{\xi},Y)^{\mathbf{m}}$
  with source $(\mathbf{v},\mathbf{w})$,
  where the $\mathbf{t}$-equivalence relation
  is understood analogously
  to the case of holomorphic multi-germs $(u,J)^{\mathbf{m}}$.
\end{rem}

%%%%%%%%%%%%%%%%%%%%%%%%%%%%%%%%%%%%%%%%%%%%%%%%%%%%%%%%%%%%%%%%%%%%%%

\subsection{Universal relations\label{univreldefinition}}

Denote by
$\big((\Sigma,\Gamma)\times(M,L)\big)^{(\mathbf{t})}$
the smooth $\mathbf{t}$-jet space
of maps $(\Sigma,\Gamma)\ra(M,L)$.
The jets with sources in
$\Gamma$ are computed via relative charts.
Because the $r$-th Taylor polynomial
of a $J$-holomorphic map
is completely determined
by the holomorphic $r$-jet
and the $(r-1)$-jet of $J$
there is a well defined map
\[
j_{(\mathbf{z},\mathbf{x})}^{\mathbf{t}}(u,J)
\longmapsto
\big(
j_{(\mathbf{z},\mathbf{x})}^{\mathbf{t}}u,
j_{u(\mathbf{z},\mathbf{x})}^{\mathbf{t}-\mathbf{1}}J^{\mathbf{m}}
\big).
\]
Because of the local description
given in the proof of
Proposition \ref{univdiffstr}
and Remark \ref{tangentvector}
this map is an embedding
\[
\NN\co
\Jet^{\mathbf{t}}
\lra
\big(
(\Sigma,\Gamma)\times(M,L)
\big)^{(\mathbf{t})}\times
E^{(\mathbf{t}-\mathbf{1})},
\]
whose projection to the second factor
is a surjective submersion.

\begin{defi}
  \label{def:hointrel}
  A submanifold $\RR$ of $\Jet^{\mathbf{t}}$
  is called a
  {\bf higher order intersection relation}
  if the image under the natural map $\NN$
  is a product manifold with second factor
  equal to $E^{(\mathbf{t}-\mathbf{1})}$.
  A relation $\RR$ is called {\bf source-free}
  if in all local fibre charts
  (as obtained in the proof of Proposition \ref{univdiffstr})
  $\RR$ is a product manifold
  with first factor equal to $(U,U')^{\mathbf{m}}$.
\end{defi}

In other words, a submanifold $\RR$ of $\Jet^{\mathbf{t}}$
is a higher order intersection relation
provided the definition of $\RR$ is \emph{$J$-independent}.
Because the first factor in a local description
of $\Jet^{\mathbf{t}}$ corresponds to the source space $\SS$
a relation $\RR$ is source-free if it
does not put restrictions on the marked points.
For example a relation
that fixes the marked points is never source-free.

With the previous discussion we obtain:

\begin{prop}
  \label{codimofr}
  Let $\RR$ be a higher order intersection relation.
  The intersection $R$ of $\RR$
  with any fibre over $E^{(\mathbf{t}-\mathbf{1})}$
  is a submanifold and the codimension of $\RR$ is
  \[\|\mathbf{m}\|+n(\|\mathbf{m}\|+\|\mathbf{t}\|)-\dim R.\]
  The intersection $R$ of a source-free relation $\RR$
  with any fibre over $\SS\times E^{(\mathbf{t}-\mathbf{1})}$
  is a submanifold and the codimension of $\RR$ is
  \[n(\|\mathbf{m}\|+\|\mathbf{t}\|)-\dim R.\]
\end{prop}

%%%%%%%%%%%%%%%%%%%%%%%%%%%%%%%%%%%%%%%%%%%%%%%%%%%%%%%%%%%%%%%%%%%%%%

\subsection{Universal moduli space\label{univmodspaceisbanmanifold}}

Let $\Sigma$ be a Riemann surface
which is either closed or
compact with boundary $\Gamma$.
Let $\RR$ be a higher order intersection relation
of order $\mathbf{t}$ and
let $A$ be a relative $2$-homology class in $(M,L)$.
The {\bf universal moduli space} $\UU$
by definition is the set of all tuples
$(u,J,\mathbf{z},\mathbf{x})$,
where $(\mathbf{z},\mathbf{x})\in\SS$ are marked points,
$J\in\II$ is an almost complex structure as in
Section \ref{acstandfloer},
and $u$ is a $J$-holomorphic map
$(\Sigma,\Gamma)\ra(M,L)$ that represents the class $A$,
such that
\begin{itemize}
\item $j_{(\mathbf{z},\mathbf{x})}^{\mathbf{t}}(u,J)\in\RR$,
\item $u$ is {\bf simple}
  (i.e., the open set of injective points is dense)
  and {\bf simple along the boundary}
  (i.e., the open set of injective points
  of $u_{|\Gamma}$ is dense in $\Gamma$),
  see \cite[Corollary 8.5]{zehm13},
\item $u(\Sigma)$ is contained
  in the perturbation domain,
  see Section \ref{acstandfloer}.
\end{itemize}

\begin{prop}
  \label{unimodspacerel}
  The universal moduli space $\UU$
  is a separable Banach manifold.  
\end{prop}

\begin{proof}
  The universal moduli space
  for the empty relation $\UU_{\emptyset}$
  (i.e., $\mathbf{m}=(0,0)$)
  is a separable Banach manifold,
  see \cite[Chapter 3]{mcsa04}.
  In the presence of a relation $\RR$
  we consider the jet extension map
  \[
  (u,J,\mathbf{z},\mathbf{x})
  \longmapsto
  j_{(\mathbf{z},\mathbf{x})}^{\mathbf{t}}(u,J).
  \]
  The preimage of $\RR$ under
  \[
  j^{\mathbf{t}}\co
  \UU_{\emptyset}\times\SS\lra\Jet^{\mathbf{t}}
  \]
  equals $\UU$.
  We consider a tangent vector of
  $\Jet^{\mathbf{t}}$ at
  $j_{(\mathbf{z},\mathbf{x})}^{\mathbf{t}}(u,J)$
  that is transverse to $\RR$.
  By $J$-independence of $\RR$
  we are free to assume
  that the tangent vector is taken
  with respect to the trivial infinitesimal
  almost complex structure,
  see Remark \ref{tangentvector}.
  This means that for
  $(\mathbf{v},\mathbf{w})\in T_{(\mathbf{z},\mathbf{x})}\SS$
  and a local holomorphic section
  $\xi$ of $(u^*TM,u^*TL)$
  near $\{\mathbf{z},\mathbf{x}\}$
  the tangent vector is given by
  \[
  j_{(\mathbf{v},\mathbf{w})}^{\mathbf{t}}(\tilde{\xi},0).
  \]
  The holomorphic curve $u$
  has the annulus property
  (see \cite[Theorem 1.1]{zehm13})
  and the half-annulus property,
  see \cite[Theorem 1.3 and Corollary 9.5]{zehm13}.
  Therefore,
  as on \cite[p.\ 63]{mcsa04}
  $\xi$ extends to a smooth global section $\xi$
  such that $(\xi,Y)$ is tangent to $\UU_{\emptyset}$
  for an infinitesimal almost complex structure
  $Y\in\II$,
  see \cite[Exercise 3.4.5]{mcsa04}
  or \cite[Lemma 4.45]{wen10b}.
  Because the universal $\mathbf{t}$-jet space
  has finite dimension
  the $\mathbf{t}$-jet extension map
  is transverse to $\RR$.
  The claim follows with \cite[Section II.2]{lang99}.
\end{proof}

\begin{rem}
  \label{wellknowncase}
  If $\Gamma=\emptyset$,
  Proposition \ref{unimodspacerel} follows from
  \cite[Lemmata 6.5, 6.6]{ciemoh07}
  and \cite[Chapter 3]{mcsa04}.
  Moreover,
  it is enough to require the holomorphic curves
  to intersect the perturbation domain non-trivially.
  If $\Gamma\neq\emptyset$ and $\|\mathbf{t}\|=0$,
  Proposition \ref{unimodspacerel}
  follows with \cite[Chapter 3]{mcsa04}.
  In that case the assumptions can be relaxed:
  Each connected component of $\Sigma$
  has an injective point of $u$
  which is mapped to the perturbation domain.
  The reader is referred to Appendix
  \ref{locpertapp}.
\end{rem}

%%%%%%%%%%%%%%%%%%%%%%%%%%%%%%%%%%%%%%%%%%%%%%%%%%%%%%%%%%%%%%%%%%%%%%

\subsection{Generic perturbation\label{subsec:genpert}}

By the argument
in Section \ref{univmodspaceisbanmanifold}
the universal moduli space
$\UU$ is a submanifold in $\UU_{\emptyset}\times\SS$
with co-dimension given by Proposition \ref{codimofr}.
Moreover,
the projection $\UU_{\emptyset}\ra\II$
is a smooth Fredholm map
of index $\mu(A)+n\chi(\Sigma)$,
where $\mu$ denotes the Maslov index
and $\chi$ the Euler characteristic,
see \cite[Chapter 3 and Appendix C]{mcsa04}.
Therefore,
the induced map
$\UU\subset\UU_{\emptyset}\times\SS\ra\II$
is Fredholm of index
$\mu(A)+n\chi(\Sigma)+\|\mathbf{m}\|-\codim\RR$.
By the Sard-Smale theorem
the set of regular values
is of second Baire category in $\II$
so that by the implicit function
theorem the preimage $\UU_J$
of a regular value $J\in\II$
is a manifold of dimension given by the index,
cf.\ \cite[Appendix A]{mcsa04}.
We call $\UU_J$ the $\RR${\bf -moduli space}
and write $\UU_{\emptyset,J}$
in the case of the empty relation.
Notice that
$J$ is a regular value
if $J$ is regular in the sense
of \cite[Definition 3.1.4]{mcsa04},
i.e.,
the linearized Cauchy-Riemann operator at
$u\in\UU_{\emptyset,J}$
is transverse to the zero-section
for all $u\in\UU_{\emptyset,J}$,
and the jet extension map
$j^{\mathbf{t}}\co\UU_{\emptyset,J}\times\SS\ra\jet^{\mathbf{t}}$
is transverse to $R$.
We call $J$ $\RR${\bf -regular} or {\bf generic}.

\begin{thm}
  \label{mainthmfirstfrom}
  The set of all
  $\RR$-regular almost complex structures
  is of second Baire category in $\II$
  (and therefore dense in $\JJ$).
  For all $\RR$-regular $J\in\II$
  the $\RR$-moduli space
  is a manifold of dimension
  \[
  \mu(A)+
  n\big(
  \chi(\Sigma)-\|\mathbf{m}\|-\|\mathbf{t}\|
  \big)+
  \dim R.
  \]
  If $\RR$ is source-free the dimension is
  \[
  \mu(A)+
  n\big(\chi(\Sigma)-\|\mathbf{t}\|\big)+
  (1-n)\|\mathbf{m}\|+
  \dim R.
  \]
\end{thm}

After an additional perturbation of $J\in\II$
smoothness of $\RR$-moduli spaces holds for holomorphic maps
that have an injective point on each connected component of $\Sigma$,
see Theorem \ref{transforallsomewinjcurves} below.
Moreover, Theorem \ref{transforallsomewinjcurves}
will allow us to work with the $C^{\infty}$-topology.

%%%%%%%%%%%%%%%%%%%%%%%%%%%%%%%%%%%%%%%%%%%%%%%%%%%%%%%%%%%%%%%%%%%%%%
%%%%%%%%%%%%%%%%%%%%%%%%%%%%%%%%%%%%%%%%%%%%%%%%%%%%%%%%%%%%%%%%%%%%%%

\section{Enumerative relations\label{enumrel}}

We assume the dimension of $M$ to be greater or equal than $4$.

%%%%%%%%%%%%%%%%%%%%%%%%%%%%%%%%%%%%%%%%%%%%%%%%%%%%%%%%%%%%%%%%%%%%%%

\subsection{A priori perturbation and local behaviour\label{aprioripertandlocb}}

We consider holomorphic curves $u$
such that on each connected component
there exists an injective point
which is mapped into the perturbation domain
under $u$.
By Remark \ref{wellknowncase}
and \cite[Chapter 3]{mcsa04}
(cf.\ Appendix \ref{locpertapp})
the following moduli problem
can be assumed to be transverse
for a generic choice of $J\in\JJ$:
\begin{itemize}
\item $[u]=A$,
\item $u$ has $n_0$ interior,
  $n_1$ boundary,
  and $\ell$ mixed double points,
  i.e.,
  $u$ is subject to the intersection relation
  \[
  \big(\Delta_M\big)^{n_0}\times
  \big(\{(p,q)\in M\times L\;|\;p=q\}\big)^{\ell}\times
  \big(\Delta_L\big)^{n_1}
  \]
  of order
  $\|\mathbf{t}\|=0$
  and length $\|\mathbf{m}\|=(2n_0+\ell,2n_1+\ell)$.
\end{itemize}
Because $\|\mathbf{t}\|=0$
in the present situation
we can indeed work with $\JJ$
instead of $\II$
regardless whether the boundary of $\Gamma$
is empty or not.

Taking the dimension formula into account
this leads to the following bound
on the number of self-intersections
(cf.\ with \cite{bircor07,lazz11}
for similar considerations):
\[
(n-2)\|(n_0,n_1)\|+(2n-3)\ell\leq\mu(A)+n\chi(\Sigma).
\]
Therefore,
$u$ has finitely many double points provided $n\geq3$.
In particular,
$u$ is simple and simple along the boundary.
For $n=2$ there are finitely many mixed double points.
With the results in \cite{lazz00,lazz11}
the bound on the number of mixed intersection points
implies that $u$ is simple.
For an alternative argument for this
we refer to Remark \ref{aprioripertrem}
below.
However,
an even stronger result holds true.
In addition,
$u$ is simple along the boundary:

\begin{cor}
  \label{aprioripert}
  There exists a subset $\JJ_{\infty}$
  of second Baire category in $\JJ$
  such that for all $J\in\JJ_{\infty}$
  any somewhere injective
  $J$-holomorphic curve is simple
  and simple along the boundary.
\end{cor}

\begin{proof}
  $\JJ_{\infty}$ is the intersection
  of sets of second category,
  where the intersection is taken over all
  relative homology $2$-classes $A$ in $(M,L)$
  (which is known to be countable)
  and all intersection relations as above
  counted by $n_0$, $n_1$, and $\ell$.
  This proves the corollary in dimension
  $2n\geq6$.
  
  For $2n=4$ it suffices to establish simplicity
  along the boundary.
  We consider the following moduli problem
  taking Remark \ref{wellknowncase} into account:
  Let $N$ be a natural number.
  Let $\PP_N$ be a partition of $\Gamma$
  into $N$ segments
  (i.e., connected open subsets)
  of equal length.
  We equip each segment $S\in\PP_N$
  with pairwise distinct points
  $y_1,\ldots,y_N\in S$.
  The holomorphic curves $u$
  are assumed to be somewhere injective.
  We require
  \begin{itemize}
  \item $[u]=A$,
  \item there are pairwise distinct points
    $x_1,\ldots,x_N\in\Gamma$ each different from
    the $y_j$'s
    such that $u(x_j)=u(y_j)$ for all $j=1,\ldots,N$.
  \end{itemize}
  Let $\JJ_{\infty}$ be the intersection of all
  regular values in $\JJ$ taken over all $A$
  and $N$ corresponding to the moduli problems.
  Hence, the moduli spaces $\{(u,x_1,\ldots,x_N)\}$
  are cut out transversely
  for $J\in\JJ_{\infty}$
  and are of dimension $\mu(A)+2\chi(\Sigma)-N$.
  
  Arguing by contradiction
  we suppose that $u$ is a somewhere
  injective holomorphic curve
  which is not simple along the boundary.
  By \cite[Proposition 9.3]{zehm13}
  there exists a diffeomorphism
  $\varphi\co S_1\ra S_2$
  between two disjoint segments of $\Gamma$
  such that $u(x)=u\big(\varphi(x)\big)$
  for all $x\in S_1$.
  Let $N>\mu([u])+2\chi(\Sigma)$
  be a natural number such that the length
  of a segment in $\PP_N$ is smaller
  than $1/2$ times the length of $S_2$.
  Let $S\in\PP_N$ be a segment
  which is contained in $S_2$.
  In particular,
  with $x_j:=\varphi^{-1}(y_j)$
  we obtain $u(x_j)=u(y_j)$ for all $j=1,\ldots,N$.
  Therefore, $(u,x_1,\ldots,x_N)$
  is an element of one of the above moduli spaces
  which has a negative dimension.
  This contradiction proves the claim.
\end{proof}

\begin{rem}
  \label{aprioripertrem}
  A similar argument shows simplicity of $u$.
  Replace $\PP_N$ by a covering of
  open balls of
  radius $1/N$ each equipped with $N$ distinct points.
  Simplicity fails if there is a diffeomorphism
  $\varphi\co U_1\ra U_2$ between two open disjoint
  subsets of $\Sigma$ such that
  $u(z)=u\big(\varphi(z)\big)$
  for all $z\in U_1$,
  see \cite[Proposition 2.7]{zehm13}.
  Choose $N>\tfrac12(\mu([u])+2\chi(\Sigma))$
  such that there is a ball in $\PP_N$
  which is a subset of $U_2$.
  
  For reasons of beauty
  we give one more argument.
  Because generically
  there are finitely many
  mixed double points
  the weak and the strong variant
  of simplicity along the boundary
  considered in \cite{zehm13}
  are the same,
  see \cite[Corollary 8.5]{zehm13}.
  By \cite[Proposition 6.4]{zehm13}
  any holomorphic curve
  that is strongly simple along the boundary
  is in fact simple.
  The argument in the proof
  of Corollary \ref{aprioripert},
  which shows weak simplicity along the boundary,
  suffices to show simplicity.
\end{rem}

%%%%%%%%%%%%%%%%%%%%%%%%%%%%%%%%%%%%%%%%%%%%%%%%%%%%%%%%%%%%%%%%%%%%%%

\subsection{Generic multiply covered discs\label{multicoverddiscs}}

A non-constant holomorphic disc $u$
(i.e., $\Sigma=\D\subset\C$
provided with the complex structure
induced by $\rmi$)
is called {\bf multiply covered}
if there exist
a simple holomorphic disc $v$
and a holomorphic map
$\pi\co (\D,\partial\D)\ra(\D,\partial\D)$
such that $u=v\circ\pi$,
where $\pi$ is required to be
of degree $\geq2$,
continuous up to the boundary,
and to satisfy
$\pi^{-1}(\partial\D)=\partial\D$.
In \cite[Theorem B]{lazz11} Lazzarini
proved the remarkable fact
that generically all non-constant
holomorphic discs attached to $L$
are simple or multiply covered
provided $\dim M\geq6$.
We drop the restriction to the dimension.

\begin{cor}
  For generic $J\in\JJ$
  each non-constant holomorphic discs
  which is contained in the perturbation domain
  is simple or multiply covered.
\end{cor}

\begin{proof}
  Lazzarini's proof of \cite[Theorem B]{lazz11}
  is based on his decomposition theorem
  \cite[Theorem A]{lazz11}.
  Each non-constant holomorphic disc $u$
  can be cut along an embedded graph
  in $\D$ into finitely many
  simple holomorphic discs
  whose union has the same image as $u$
  and whose homology classes
  weighted with positive multiples
  add up to $[u]$.
  On \cite[p.\ 254/5]{lazz11}
  he shows that the graph equals $\partial\D$
  using (\cite[Proposition 5.15]{lazz11},
  which can replaced by)
  the following two results:
  \begin{itemize}
  \item Let $v$ be a simple holomorphic disc.
    Then there exists no diffeomorphism
    $\varphi\co S_1\ra S_2$
    between two disjoint open segments
    on $\partial\D$ such that $v=v\circ\varphi$.
  \item Let $v_1$ and $v_2$
    be simple holomorphic discs
    such that
    $v_1(\D)\not\subset v_2(\D)$ and
    $v_2(\D)\not\subset v_1(\D)$.
    Then there exists no diffeomorphism
    $\varphi\co S_1\ra S_2$
    between two disjoint open segments
    on $\partial\D$ such that $v_1=v_2\circ\varphi$.
  \end{itemize}
  Notice that $v_1,v_2$ define a somewhere injective
  holomorphic map on $\D\sqcup\D$ in the sense
  that each connected component
  has an injective point.
  An application of the argument from Corollary \ref{aprioripert}
  proves both items.
  This establishes Lazzarini's theorem for $\dim M=4$.
\end{proof}

%%%%%%%%%%%%%%%%%%%%%%%%%%%%%%%%%%%%%%%%%%%%%%%%%%%%%%%%%%%%%%%%%%%%%%

\subsection{Addendum\label{addenthm}}

Using Corollary \ref{aprioripert}
we can generalize
Theorem \ref{mainthmfirstfrom}
to somewhere injective holomorphic curves.
In view of Remark \ref{wellknowncase}
additional arguments are only required in
the case where the relation puts conditions
on the holomorphic jets
of order greater or equal than $1$ along the boundary.
In addition, Corollary \ref{aprioripert}
allows us to replace $\II$ by $\JJ$:

\begin{thm}
  \label{transforallsomewinjcurves}
  Let $\RR$ be a higher order jet relation
  of order $\mathbf{t}$.
  Then there exists a subset $\JJ_{\mathrm{gen}}$
  of second Baire category in $\JJ$
  such that for all $J\in\JJ_{\mathrm{gen}}$
  the conclusion
  of the Theorem \ref{mainthmfirstfrom}
  holds for somewhere injective
  $J$-holomorphic curves
  that are contained in the perturbation domain.
\end{thm}

\begin{proof}
  Theorem \ref{mainthmfirstfrom}
  provides a subset of
  $\RR$-regular almost complex structures
  of $\II$
  that is of second Baire category in $\II$
  and, therefore, dense in $\JJ$.
  Moreover,
  Corollary \ref{aprioripert}
  can be obtain in the weaker form
  using $\II$ instead of $\JJ$.
  This means 
  that there exists a subset of $\II$
  of second Baire category
  such that all corresponding
  somewhere injective holomorphic curves
  are simple and simple along the boundary.
  Taking the intersection of both
  second Baire category sets,
  which is dense in $\JJ$,
  Theorem \ref{transforallsomewinjcurves}
  follows with respect to the Floer-topology $\II$.
  
  The stronger formulation of the theorem,
  which allows to work with the $C^{\infty}$-topology,
  follows with the Taubes trick
  (see \cite[p.\ 52-53]{mcsa04}
  and \cite[Lemma 4.50]{wen10b})
  as it was pointed out to the author
  by the referee:
  In view of the assumptions of the theorem
  the universal moduli space $\UU$
  is understood to refer to holomorphic curves
  that are somewhere injective
  (not only simple and simple along the boundary).
  Similarly,
  we proceed with the moduli spaces $\UU_J$.
  
  Consider a sequence of compact subsets $K_{\nu}$ of $M$
  such that $K_{\nu}\subset\Int K_{\nu+1}$
  and $M=\bigcup_{\nu\in\N}K_{\nu}$.
  Let $\UU_J^{\nu}$ be the subset of
  $J$-holomorphic curves
  $(u,\mathbf{z},\mathbf{x})\in\UU_J$
  satisfying the following conditions:
  \begin{itemize}
  \item 
    The image $u(\Sigma)$ is contained in $K_{\nu}$.
  \item 
    $\sup_{\Sigma}|Tu|\leq\nu$.
  \item 
    On any connected component of $\Sigma$
    there exists a point $z_0$
    such that
    \[
    \inf_{z\neq z_0}
    \frac
    {\dist\big(u(z),u(z_0)\big)}
    {\dist(z,z_0)}
    \geq\frac{1}{\nu}\;.
    \]
  \item 
    The distance
    of $(\mathbf{z},\mathbf{x})$
    to $(\Sigma\times\Gamma)^{\mathbf{m}}\setminus\SS$
    is greater or equal to $1/\nu$.
  \item 
    For a suitable metric on $\Jet^{\mathbf{t}}$
    the distance of $j_{(\mathbf{z},\mathbf{x})}^{\mathbf{t}}(u,J)$
    to the boundary of $\RR$
    is greater or equal to $1/\nu$.
  \end{itemize}
  Define $\JJ_{\nu}$ to be the set
  of almost complex structures
  $J\in\JJ$ that are $\RR$-regular
  for all holomorphic curves in $\UU_J^{\nu}$.  
  As explained in the first part of the proof
  $\JJ_{\nu}$ is dense in $\JJ$.
  Because the set $\UU_J^{\nu}$
  is compact the set $\JJ_{\nu}$
  is open as well.
  Let $\JJ_{\mathrm{gen}}$
  be the intersection
  $\bigcap_{\nu\in\N}\JJ_{\nu}$,
  which is of second Baire category in $\JJ$.
  This implies the theorem
  because the union $\bigcup_{\nu\in\N}\UU_J^{\nu}$
  equals $\UU_J$.
\end{proof}

%%%%%%%%%%%%%%%%%%%%%%%%%%%%%%%%%%%%%%%%%%%%%%%%%%%%%%%%%%%%%%%%%%%%%%

\subsection{Uniform bounds on self-intersection points\label{secadjineq}}

The group of conformal automorphisms $G$
of $(\Sigma,\Gamma)$
acts on the $\RR$-moduli space
for any source-free relation $\RR$ by
\[
g\cdot (u,\mathbf{z},\mathbf{x})=
\big(
u\circ g^{-1},g(\mathbf{z}),g(\mathbf{x})
\big).
\]
The quotient
\[
\MM_R=\UU_J/G
\]
is a smooth manifold
provided the action is free.
The dimension equals
\[
\mu(A)+
n\big(\chi(\Sigma)-\|\mathbf{t}\|\big)+
(1-n)\|\mathbf{m}\|+
\dim R-d,\]
where $d$ is the dimension of $G$.
We only will consider the case
where the action is free.
Moreover, we do not consider variations
of complex structures on the surface $\Sigma$
as in \cite{wen10};
but we remark that an adaptation of the method
of the present work is done by minor changes.

A consequence of the discussion
in Section \ref{aprioripertandlocb} is:

\begin{cor}
  \label{coradineq}
  For a generic choice of $J\in\JJ$
  and all $J$-holomorphic curves $u$
  which have an injective point on each
  connected component of $\Sigma$
  which are mapped into the perturbation domain by $u$
  we have:
  \[(n-2)\|(n_0,n_1)\|+(2n-3)\ell\leq\mu([u])+n\chi(\Sigma)-d.\]
  Here $n_0$ (resp.\ $n_1$, $\ell$) is the number of
  interior (resp.\ boundary, mixed) double points.
\end{cor}

In particular,
if the Maslov index vanishes and $n\geq3$
generically somewhere injective holomorphic disc-maps
are injective.

%%%%%%%%%%%%%%%%%%%%%%%%%%%%%%%%%%%%%%%%%%%%%%%%%%%%%%%%%%%%%%%%%%%%%%

\subsection{Singularities\label{seconsing}}

A non-constant holomorphic map $u$
has a {\bf singularity of order} $k$ at $z$
if $j_z^ku=0$ and $D^{k+1}u(z)\neq 0$.
By Carleman's similarity principle
(cf.\ \cite[Chapter A.6]{hoze94}
and \cite[Theorem A.2]{abb04})
$k$ is the greatest natural number
such that all partial derivatives
of $u$ at $z$ vanish up to order $k$.

For a generic choice of
compatible almost complex structure
the moduli space of unparametrized somewhere injective
(in the sense of Remark \ref{wellknowncase})
holomorphic curves
with vanishing derivatives at $\mathbf{m}$ points
up to order $\mathbf{k}$
has dimension
\[\mu(A)+n\big(\chi(\Sigma)-\|\mathbf{k}\|\big)+\|\mathbf{m}\|-d.\]
As in Corollary \ref{coradineq} we obtain,
cf.\ \cite[Corollary 3.17]{wen10} and \cite{ohzh09,oh11}:

\begin{cor}
  \label{corsing}
  For generic $J\in\JJ$
  and all somewhere injective $J$-holomorphic curves $u$
  that are contained in the perturbation domain
  we have:
  \[n\|\mathbf{k}\|-\|\mathbf{m}\|\leq\mu([u])+n\chi(\Sigma)-d.\]
  Here $\mathbf{k}$ is the order of singularities of $u$ at
  $\mathbf{m}$ points.
\end{cor}

For example, generically,
all somewhere injective holomorphic discs
with Maslov index less or equal than $1$
are immersed.

\begin{rem}
  The map $(u,\mathbf{z},\mathbf{x})\mapsto u$
  which is defined on the moduli space
  of curves as in Corollary \ref{corsing}
  is an immersion into the moduli space of unconstrained curves.
  This is because the kernel of the linearization
  $(\xi,\mathbf{v},\mathbf{w})\mapsto\xi$
  is given by all $(\mathbf{v},\mathbf{w})\in T\SS$
  satisfying
  $j_{(\mathbf{z},\mathbf{x})}^{\mathbf{k}}Tu(\mathbf{v},\mathbf{w})=0$.
\end{rem}

%%%%%%%%%%%%%%%%%%%%%%%%%%%%%%%%%%%%%%%%%%%%%%%%%%%%%%%%%%%%%%%%%%%%%%

\subsection{Example\label{secexample}}

Consider a rational split Lagrangian submanifold
$L=S^1\times L'$ in $\R^2\times\R^{2n-2}$
with area spectrum $\pi\Z$.
For generic compatible almost complex structures $J$
which equal $\rmi$ outside a large ball $B$
we consider holomorphic discs
representing the class $[D\times\{*\}]$,
which are simple,
see \cite{lazz00,lazz11}.
By Corollary \ref{corsing}
there is at most one singularity
on each disc,
which is simple and
lies on the boundary.
Indeed, the right hand side of the inequality
in Corollary \ref{corsing} equals $n-1$.
Because $\|\mathbf{m}\|\leq\|\mathbf{k}\|$
this implies that $\|\mathbf{m}\|\leq 1$.
Hence, $\|\mathbf{k}\|\leq 1$ too.
Therefore, the zero-dimensional moduli space
of discs with a simple boundary singularity
$\MM_{\partial\text{-}\mathrm{sing}}$
embeds via $[u,x_1]\mapsto [u]$
into the moduli space of all discs $\MM$.
By Gromov compactness $\MM$ is compact,
see \cite{fra08,grom85}.
Therefore, there are finitely many discs
with singularities.

The number is even.
To prove this
consider a generic path $J(t)$
of almost complex structures
that connects $\rmi$ with $J$
and that are equal to $\rmi$ outside $B$.
Let 
$\WW_{\partial\text{-}\mathrm{sing}}$
be the moduli space
of unparametrized holomorphic discs
with a boundary singularity that 
are holomorphic with respect to
one of the $J(t)$'s.
With \cite[p.\ 43]{mcsa04},
the regularity arguments of the present work,
and Gromov compactness \cite{fra08,grom85}
$\WW_{\partial\text{-}\mathrm{sing}}$
is a $1$-dimensional cobordism.
Because all $\rmi$-holomorphic discs in the class
$[D\times\{*\}]$
are embedded we obtain with the argument from
\cite[Remark 3.2.8]{mcsa04}
that the boundary of 
$\WW_{\partial\text{-}\mathrm{sing}}$
equals
$\MM_{\partial\text{-}\mathrm{sing}}$.
The number of boundary components is even.

%%%%%%%%%%%%%%%%%%%%%%%%%%%%%%%%%%%%%%%%%%%%%%%%%%%%%%%%%%%%%%%%%%%%%%

\subsection{Embeddings\label{subsecemb}}

We consider somewhere injective
holomorphic curves
contained in the perturbation domain
that represent a homology class $A$
(and assume that the biholomorphic
reparametrisation group acts freely
on the moduli spaces under consideration).
Generically, the corresponding
moduli space $\MM$
is a manifold of dimension
$\mu(A)+n\chi(\Sigma)-d$.
We can assume that the moduli spaces
of curves
which additionally have a singularity
in the interior or
on the boundary
are manifolds of dimension
\begin{itemize}
\item $\dim\MM_{\mathrm{sing}}=\dim\MM+2(1-n)$
\item $\dim\MM_{\partial\text{-}\mathrm{sing}}=\dim\MM+1-n$
\end{itemize}
these with interior, mixed, and boundary
double points are manifolds of dimension
\begin{itemize}
\item $\dim\MM_{\mathrm{inter}}=\dim\MM+2(2-n)$
\item $\dim\MM_{\mathrm{mix}}=\dim\MM+3-2n$
\item $\dim\MM_{\partial\text{-}\mathrm{inter}}=\dim\MM+2-n$
\end{itemize}
and that the moduli space of curves
that intersect $L$ in an interior point
is a manifold of dimension
\begin{itemize}
\item $\dim\MM_L=\dim\MM+2-n$.
\end{itemize}
Each of the above moduli spaces admits a smooth map
into $\MM$ induced by forgetting the marked points.
The complement of the images is the subset of
{\bf perfectly embedded} curves.
By Sard's theorem
the union of the critical values has measure zero.
Because the closure of the images of the moduli spaces
subject to a pure intersection relation
is contained in the union with
the images of $\MM_{\mathrm{sing}}$
and $\MM_{\partial\text{-}\mathrm{sing}}$
we obtain:

\begin{cor}
  Generically,
  the set of all perfectly embedded
  holomorphic curves
  in $\MM$ is open and dense
  provided $\dim M\geq6$.
\end{cor}

An immersed holomorphic curve is called {\bf clean}
if there is no mixed double point
and all double points
and all interior intersections with $L$
are simple and transverse.
With similar arguments we get:

\begin{cor}
  If $\dim M=4$
  the set of all clean immersions in $\MM$
  generically is open and dense.
\end{cor}

\begin{appendix}
  \section{Local perturbations\label{locpertapp}}

  As stated in Remark \ref{wellknowncase}
  the requirements on the perturbation domain
  (in order to provide the moduli space of holomorphic maps
  subject to pointwise constraints
  with a smooth structure)
  can be weakened.
  Consider the cases:
  \begin{itemize}
  \item 
    The holomorphic maps are simple,
    defined on a closed Riemann surface,
    and the image of each connected component
    intersects the perturbation domain
    non-trivially.
    The intersection relation is allowed
    to be of higher order.
  \item 
    The domain $\Sigma$ is allowed to
    have non-empty boundary $\Gamma$,
    the intersection relation is required
    to be of order $\|\mathbf{t}\|=0$,
    and on each connected component
    there exists an injective point for
    the holomorphic map $u$ that is mapped by $u$
    into the perturbation domain,
    see Section \ref{acstandfloer}.
    It can be assumed that the distinguished injective points
    are contained in 
    $\Sigma\setminus\{\mathbf{z},\mathbf{x}\}$
    because the set of injective points is open.
  \end{itemize}

  \begin{thm}
    \label{transforlocpertapp}
    There exists a subset $\JJ_{\mathrm{gen}}$
    of second Baire category in $\JJ$
    such that for all $J\in\JJ_{\mathrm{gen}}$
    the conclusion
    of the Theorem \ref{mainthmfirstfrom}
    holds for $J$-holomorphic curves
    as described.
  \end{thm}

  \begin{proof}
    We claim that the corresponding universal
    moduli space $\UU$ is a separable Banach manifold.
    As in the proof of Proposition \ref{unimodspacerel}
    we consider a tangent vector
    \[
    j_{(\mathbf{v},\mathbf{w})}^{\mathbf{t}}(\tilde{\xi}_1,0)
    \]
    represented by a local holomorphic section
    $\xi_1$ of $(u^*TM,u^*TL)$
    near $\{\mathbf{z},\mathbf{x}\}$.
    It can be assumed
    that $\xi_1$ is extended to a global smooth section.
    We will construct a further smooth section $\xi_2$
    and an infinitesimal almost complex structure $Y\in T_J\II$
    such that $(\xi,Y)$ with $\xi=\xi_1+\xi_2$
    is tangent to $\UU_{\emptyset}$
    and
    \[
    j_{(\mathbf{v},\mathbf{w})}^{\mathbf{t}}(\tilde{\xi},Y)=
    j_{(\mathbf{v},\mathbf{w})}^{\mathbf{t}}(\tilde{\xi}_1,0).
    \]
    This will show that
    the $\mathbf{t}$-jet extension map
    is transverse to $\RR$
    because the universal $\mathbf{t}$-jet space
    is finite-dimensional.
    That $\UU$ is a separable Banach manifold
    follows then with \cite[Section II.2]{lang99}.
    
    We construct the pair $(\xi_2,Y)$.
    Denote by $\nabla$ the Levi-Civita connection
    of the metric $\langle\,.\,,.\,\rangle=\omega(\,.\,,J\,.\,)$.
    Consider the linearized Cauchy-Riemann operator
    \[
    D_u\xi=
    \nabla\xi+J(u)\,\nabla\xi\circ\rmi
    +\nabla_{\xi}J(u)\,Tu\circ\rmi,
    \]
    which maps smooth sections $\xi$ of $(u^*TM,u^*TL)$
    into smooth $J$-antilinear $1$-forms on $\Sigma$
    with values in $u^*TM$,
    cf.\ \cite{wen10b}.
    Set
    \[
    \eta_1:=-D_u\xi_1
    \]
    and observe that $\eta_1$ vanishes in a
    neighbourhood of $\{\mathbf{z},\mathbf{x}\}$.
    We claim that for $k=\max\{\mathbf{t}\}$
    there exists $(\xi_2,Y)$ such that
    \begin{itemize}
    \item 
      $\xi_2$ vanishes at least to the $k^{\mathrm{th}}$
      order on the marked points $\{\mathbf{z},\mathbf{x}\}$,
    \item 
      $Y\in T_J\II$ has support in the complement of
      $u(\mathbf{z},\mathbf{x})$, and
    \item 
      $D_u\xi_2=\eta_1-Y(u)\,Tu\circ\rmi$.
    \end{itemize}
    Any solution $(\xi_2,Y)$ of the problem
    has the required properties.
    Therefore, in order to provide $\UU$
    with the desired smooth structure    
    it suffices to solve the problem.
    
    If $k=0$ this is the content of the proof of
    \cite[Lemma 3.4.3]{mcsa04}.
    Otherwise, if $k\geq1$ and $\Gamma=\emptyset$,
    we argue as follows:  
    Denote by $A$ the Banach space of sections
    of $(u^*TM,u^*TL)$ of Sobolev class $W^{k+1,p}$, $p>2$,
    that vanish on the marked points
    $\{\mathbf{z}\}$ up to
    order $k$.
    Denote by $B$ the Banach space of
    $J$-antilinear $1$-forms on $\Sigma$
    with values in $u^*TM$ of Sobolev class $W^{k,p}$
    that vanish on the marked points
    $\{\mathbf{z}\}$ up to order $k-1$.
    Let $I\subset T_J\II$ be the space of
    infinitesimal almost complex structures
    that have support in a prescribed neighbourhood $U$
    of the image of the distinguished injective points
    of $u$ such that $U$ is disjoint from
    $u(\mathbf{z})$.
    Let $F\co A\times I\ra B$ be the operator
    given by $F(\xi,Y)=D_u\xi+Y(u)\,Tu\circ\rmi$.
    By the arguments of \cite[Lemma 6.6]{ciemoh07}
    and Carleman's similarity principle
    the operator $F$ is onto.
    Hence,
    $\UU$ is a separable Banach manifold.
    
    To finish the proof one follows the arguments
    of Section \ref{subsec:genpert} and
    Theorem \ref{transforallsomewinjcurves}.
    In order to apply the Taubes trick
    one replaces the condition on the injective points
    by the requirement:
    On any connected component of
    $\Sigma\setminus\Sigma_{1/\nu}$,
    where we denote by $\Sigma_{1/\nu}$
    the set of points on
    $\Sigma$ that have distance smaller than $1/\nu$
    to $\{\mathbf{z},\mathbf{x}\}$,
    there exists a point $z_0$
    such that
    \[
    \inf_{z\neq z_0}
    \frac
    {\dist\big(u(z),u(z_0)\big)}
    {\dist(z,z_0)}
    \geq\frac{1}{\nu}
    \]
    and the distance of $u(z_0)$
    to the boundary of the perturbation domain
    is greater or equal to $1/\nu$.
  \end{proof}

\end{appendix}

%%%%%%%%%%%%%%%%%%%%%%%%%%%%%%%%%%%%%%%%%%%%%%%%%%%%%%%%%%%%%%%%%%%%%%
%%%%%%%%%%%%%%%%%%%%%%%%%%%%%%%%%%%%%%%%%%%%%%%%%%%%%%%%%%%%%%%%%%%%%%

\begin{ack}
  The research presented in this article
  generalizes parts of my doctoral thesis \cite{zehm08}.
  I am grateful to Matthias Schwarz for suggesting the topic
  and the unknown referee for 
  drawing my attention to the Taubes trick.
  I would like to thank Alberto Abbondandolo and Stefan Suhr
  for intensive discussions about the analysis
  behind the arguments of the present work.
  I would like to thank Georgios Dimitroglou Rizell,
  Jonathan David Evans, and Samuel Lisi 
  for pointing out the potential applications
  to (open) Gromov-Witten and (relative) symplectic field theory.
  This motivated parts of this work.
  Further,
  I would like to thank Hansj\"org Geiges
  for explaining to me how to use jets
  and for illuminating discussions about
  the correct use of the contraposition.
\end{ack}

%%%%%%%%%%%%%%%%%%%%%%%%%%%%%%%%%%%%%%%%%%%%%%%%%%%%%%%%%%%%%%%%%%%%%%
%%%%%%%%%%%%%%%%%%%%%%%%%%%%%%%%%%%%%%%%%%%%%%%%%%%%%%%%%%%%%%%%%%%%%%

\end{document}